\documentclass[11pt,a4paper]{amsart}
\usepackage{amsfonts,amsmath,amssymb,amsthm}
\usepackage{times}
\usepackage{color}
\usepackage{esint}
\usepackage{pstricks}
\usepackage{hyperref}

\numberwithin{equation}{section}

\renewcommand{\epsilon}{\varepsilon}

\renewcommand{\Re}{{\ensuremath{\mathrm{Re\,}}}}

\DeclareSymbolFont{SY}{U}{psy}{m}{n}
\DeclareMathSymbol{\emptyset}{\mathord}{SY}{'306}

\DeclareMathOperator{\Ran}{Ran} \DeclareMathOperator{\Ker}{Ker}
 \DeclareMathOperator{\Dom}{Dom}
\DeclareMathOperator{\sign}{sign} \DeclareMathOperator{\grad}{grad}

\DeclareMathOperator{\dom}{Dom} \DeclareMathOperator{\ran}{Ran}

\renewcommand{\div}{\mathrm{div}\,}

\DeclareMathOperator{\spec}{spec} 
\DeclareMathOperator{\specess}{spec_{ess}}

\DeclareMathSymbol{\newtimes}{\mathbin}{SY}{'264}

\newcommand{\dist}{\mathrm{dist}}

\newcommand{\R}{\mathbb{R}}

\newcommand{\C}{\mathbb{C}}
\newcommand{\Z}{\mathbb{Z}}

\newcommand{\EE}{\mathsf{E}}

\newcommand{\fK}{\mathfrak{K}}

\newcommand{\fa}{\mathfrak{a}}
\newcommand{\fb}{\mathfrak{b}}

\newcommand{\ft}{\mathfrak{t}}
\newcommand{\fs}{\mathfrak{s}}
\newcommand{\fv}{\mathfrak{v}}

\newcommand{\cB}{{\mathcal B}}

\newcommand{\cH}{{\mathcal H}}

\newcommand{\cK}{{\mathcal K}}
\newcommand{\cL}{{\mathcal L}}

\newcommand{\cO}{{\mathcal O}}

\newcommand{\Rey}{\mathrm{Re}}

\newcommand{\St}{\mathrm{St}_*}
\newcommand{\vv}{{v_{\ast}}}
\newcommand{\rlad}{\mathrm{Re}^*}
\newcommand{\rr}{\Re_\ast}

{\bf}{\it}
{\bf}{\it}
{\bf}{\it}
{\bf}{\it}
{\bf}{\it}
{\bf}{\it}

{\bf}{\it}

\newtheorem{theorem}{Theorem}[section]{\bf}{\it}
\newtheorem{proposition}[theorem]{Proposition}{\bf}{\it}
{\bf}{\it}
{\it}{\rm}
\newtheorem{lemma}[theorem]{Lemma}{\bf}{\it}
\newtheorem{remark}[theorem]{Remark}{\it}{\rm}
{\bf}{\it}
{\bf}{\it}
{\bf}{\it}
{\bf}{\it}
{\bf}{\it}

\title[ The Tan $2 \Theta$-theorem in fluid dynamics]{ The Tan $2 \Theta$-theorem in fluid dynamics}

\author[L. Grubi\v{s}i\'c]{Luka Grubi\v{s}i\'c}
\address{L.~Grubi\v{s}i\'c,
Department of Mathematics, University of Zagreb, Bijeni\v{c}ka 30,
10000 Zagreb, Croatia}
\email{luka.grubisic@math.hr}

\author[V. Kostrykin]{Vadim Kostrykin}
\address{V.~Kostrykin, FB 08 - Institut f\"{u}r Mathematik,
Johannes Gutenberg-Universit\"{a}t Mainz,
Staudinger Weg 9,
D-55099 Mainz,
Germany}
\email{kostrykin@mathematik.uni-mainz.de}

\author[K. A. Makarov]{Konstantin A.~Makarov}
\address{K.~A.~Makarov, Department of Mathematics, University of
Missouri, Co\-lum\-bia, MO 65211, USA}
\email{makarovk@missouri.edu}

\author[S.~Schmitz]{Stephan Schmitz}
\address{S.~Schmitz, Department of Mathematics, University of
Missouri, Co\-lum\-bia, MO 65211, USA}
\email{schmitzst@missouri.edu}

\author[K. Veseli\'c]{Kre\v{s}imir Veseli\'c}
\address{K.~Veseli\'c,
Fakult\"{a}t f\"{u}r Mathematik und Informatik, Fernuniversit\"{a}t Hagen, Postfach 940,
D-58084 Hagen, Germany} \email{kresimir.veselic@fernuni-hagen.de}


\dedicatory{Dedicated with great pleasure to Eduard Tsekanovskii  at the  occasion  of his 80th birthday}


\copyrightinfo{2017}{L.~Grubi\v{s}i\'c, V.~Kostrykin, K.~A.~Makarov, S.~Schmitz,
K.~Veseli\'c}
\begin{document}

\begin{abstract}
We show that  the generalized Reynolds number (in fluid dynamics) introduced by Ladyzhenskaya is closely related to the  rotation  of the positive spectral subspace
of the  Stokes  block-operator in the underlying Hilbert space. We also explicitly  evaluate  the bottom of the negative  spectrum of the Stokes 
operator and prove a sharp  inequality relating the distance from the bottom of its spectrum to the origin and the length of the first positive gap.

\end{abstract}

\maketitle

\section{Introduction}

It is  generally believed  that a  steady flow  of an incompressible fluid  is stable whenever  the Reynolds number associated  with the flow is sufficiently low, while it is experimentally proven that flows become turbulent for  high Reynolds numbers (about  several hundreds   and beyond).

 Historically, the first  rigorous   quantitative  stability  result  for stationary solutions to the $2D$-Navier-Stokes equation (in  bounded domains) is due to  Ladyzhenskaya
\cite{Lad}.  Her analysis  shows that  given  a stationary solution  $v_{\text{st}}$, 
  any other solution $v$ (with smooth initial data and the same forcing) approaches   $v_{\text{st}}$ exponentially fast
 \begin{equation}\label{conv}
  v-v_{\text{st}}=\cO(e^{-\alpha t}), \quad t\to\infty,
\end{equation}
whenever the generalized Reynolds number
  $$
  \rlad=\frac{2\vv}{\nu \sqrt{\lambda_1(\Omega)}}
  $$
  is  less than one.
  Here  $\nu $ is the viscosity of the incompressible fluid,   $\lambda_1(\Omega)$ is the principal eigenvalue of the Dirichlet Laplacian in the bounded  domain $\Omega$, and 
  $\vv$ stands for the characteristic velocity of the stationary flow $v_{\text{st}}$ (see \eqref{vstar}).
	
In fact,  the rate of convergence \(\alpha\) in \eqref{conv} is  given by  \cite{Lad}, $$\alpha=\nu \lambda_1(\Omega) (1-\rlad ).$$

To better understand the functional-analytic as well as (Hilbert space) geometric aspects of the Navier-Stokes stability in any dimension, we introduce and study the  (model)  Stokes block operator,  which is  the Friedrichs extension of   the block  operator matrix  
   \begin{equation}\label{StokesMatrixint} 
S=\begin{pmatrix}-\nu{\bf\Delta}&v_*\grad\\ -v_*\div &0\end{pmatrix} 
\end{equation}  
 initially defined on the set $C_0^\infty(\Omega)^n\oplus  C^\infty(\Omega)$  of  infinitely differentiable vector-valued functions  in the Hilbert space $\cH=L^2(\Omega)^n\oplus L^2(\Omega)$, $n\ge 2$ (cf.~ \cite{Atkinson, FFMM}).

One of the principal results of the current  paper links the Ladyzhenskaya-Reynolds number $\rlad$ to 
the norm of the operator angle $\Theta$  between the positive subspace   of the Stokes operator and  the positive subspace of its  diagonal part (see \cite{Davis:Kahan,KMM:2,Seelmann}  for the concept of an operator angle).
That is, the following  {\sc Tan $2 \Theta$ Theorem in  fluid dynamics},  
\begin{equation*}\label{Tan2Thetaintro}
\tan 2\|\Theta\|\leq  \rlad,
\end{equation*}
holds (see Theorem \ref{osn3}). 

The essence of this estimate is the remarkable fact that {\it the magnitude of the Reynolds number limits the rotation of the spectral subspaces of the block Stokes operator.}

We also show that   the lowest positive eigenvalue $\lambda_1(S)$ of  the Stokes operator  $S$ and the bottom of its negative (essential) spectrum
satisfy the   inequality
$$|\inf \spec (S)|\le \frac14 \left [ \rlad\right ]^2 \lambda_1(S),
$$
which is  asymptotically sharp as $\nu \to \infty$ or $\vv \to 0$.

 In particular,  the Ladyzhenskaya (2D-) stability hypothesis $\rlad<1$  yields the following \\
 {\sc Stability Laws }:
 \begin{itemize}
 \item[$\bullet$] 
 the relative spectral shift $\delta$ defined as ratio of the shift of the spectrum from the origin to the left to the length of the spectral gap of the Stokes operator
is bounded by  
 $$
\delta =\frac{|\inf \spec (S)|}{ \lambda_1(S)}<\frac{1}{4};$$

 \item[$\bullet$]
  the maximal rotation  angle $\|\Theta\|$  
 between the positive subspaces of the perturbed and unperturbed Stokes operators  is bounded by
\begin{equation*}
\label{soso}
  \|\Theta\|<\frac\pi 8;
  \end{equation*}
\item[$\bullet$]
 the Friedrichs extension of the block operator matrix 
\begin{equation*}\label{PS} 
T=\begin{pmatrix}-\nu{\bf\Delta}-\frac12 \nu \lambda_1(\Omega)&\vv\grad\\ -\vv\div &\frac12 \nu \lambda_1(\Omega)\end{pmatrix}
\end{equation*} is positive definite (via the geometric variant of the Birman-Schwinger principle for off-diagonal  perturbations \cite[Corollary 3.4]{pap:2}).

 \end{itemize}
We also observe that the Ladyzhenskaya decay exponent $\alpha$ provides the lower bound for $\inf\spec(T)$, 
 $$
\alpha =\nu \lambda_1(\Omega) (1-\rlad ) \le 2\cdot \inf \spec(T),
 $$ 
which is asymptotically sharp  in the sense that 
 $$
 \lim_{\rlad\downarrow 0}\frac{\alpha}{\inf \spec(T)}=2.
 $$
All that combined gives the  direct   operator-theoretic interpretation for  the 2D-Ladyzhenskaya  result   in the framework of the linearization method in  hydrodynamical stability theory.

The paper is organized as follows.

In Section 2,   the Stokes operator  \(S\) is defined as a self-adjoint operator in the Hilbert space $\cH$. In Theorem  \ref{osn2}, based on  the   quadratic numerical range  variational principle  (see Appendix \ref{Appendix}),
we obtain an  estimate for the first positive eigenvalue  and explicitly  evaluate  the  lower edge of \(S\). Theorem \ref{osn3}, the  Tan $2\Theta$-Theorem  in fluid dynamics, is deduced  from  a general rotation angle bound  obtained in  \cite{pap:2} for indefinite  forms. 
In Theorem \ref{xuxu}, we show that at low Reynolds numbers,    the qualitative spectral analysis for the Stokes operator is closely related to the  one for its  principal symbol.

In Section \ref{Sec:Motivation}, under the hypothesis that the generalized Reynolds number is less than one, we discuss the Stability Laws and provide an operator-theoretic interpretation for the Ladyzhenskaya Stability Theorem  \cite[Theorem 6.5.12]{Lad}. 
 
Appendix A contains supplementary material beyond the main scope of  the exposition and deals with the dimensional analysis of the problem in question.
 
First, we provide a  (heuristic) justification supporting the appearance of  the characteristic velocity  parameter \(v_*\) in the  definition of the Stokes operator.

Next, as a result of the dimensional analysis, we naturally arrive at dimensionless variables such as the  generalized   Reynolds  and Strouhal type numbers, see \eqref{VLnu} and \eqref{ReSt} for their definition.
We also  show that 
 at low Reynolds numbers, their  product  and  ratio is proportional to the distance from the bottom of the spectrum of the Stokes operator to the origin and  the length of spectral gap of the diagonal part of the Stokes operator, respectively
 (see \eqref{tor1} and  \eqref{tor2}). This observation is illustrated in the Strouhal-Reynolds-Rotation angle diagram Fig. 2.

In Appendix \ref{Appendix}, we briefly recall representation theorems for indefinite (saddle-point) forms and provide necessary information on the properties of their  quadratic numerical range (cf. \cite{T} for the concept of quadratic numerical ranges for operator matrices). 
\vspace{12pt}

We adopt the following notation. 
In the Hilbert space \(\cH\) we use the scalar product \(\langle \;\cdot\;,\;\cdot\;\rangle\) semi-linear in the first and linear in the second component.
 $I_{\fK}$ denotes the identity operator on a Hilbert space
$\fK$, where we frequently omit the subscript. 
Given a self-adjoint  operator \(S\) and a  Borel set \(M\) on the real axis, the corresponding   spectral projection is denoted by \(\EE_S(M)\).

Given an orthogonal decomposition $\fK_0\oplus\fK_1$ of the Hilbert
space $\fK$ and dense subsets $\cK_i\subset\fK_i$, $i=0,1$, by
$\cK_0\oplus\cK_1$ we denote a subset of $\fK$ formed by the vectors
$\begin{pmatrix} x_0 \\ x_1 \end{pmatrix}$ with $x_i\in\cK_i$, $i=0,1$.

\section{The Stokes Operator}\label{sec:Stokes}

Assume that $\Omega$  is a bounded \(C^2\)-domain in \(\R^n$, $n\geq 2\),  \({\bf\Delta}=\Delta \cdot I_n\) is the vector-valued Dirichlet Laplacian,  with $I_n$  the identity operator  in $\C^n$, and  $\nu>0$ and $\vv\ge 0$ are parameters.
In the direct sum of Hilbert spaces 
 $$\cH=\cH_+\oplus\cH_-,$$  where 
$\cH_+=L^2(\Omega)^n$ and $\cH_-= L^2(\Omega)$ stand for  the ``velocity'' and ``pressure'' subspaces, respectively,
consider  the Stokes block operator matrix (cf. \cite{FFMM}) given by
\begin{equation}\label{StokesMatrix} 
\begin{pmatrix}-\nu{\bf\Delta}&\vv\grad\\ -\vv\div &0\end{pmatrix}.
\end{equation}

We introduce a self-adjoint realization  \(S=S(\nu,\vv)\) of \eqref{StokesMatrix}
as a unique self-adjoint operator  
associated with the symmetric sesquilinear (saddle-point)  form 
\[\begin{aligned}
& \fs[v\oplus p,u\oplus q]\\& = \nu\sum_{j=1}^n \int_\Omega \langle D_j v(x),D_ju(x)\rangle_{\mathbb{C}^n}dx-\vv\int_\Omega\overline{\div v(x)} q(x)dx-\vv\int_\Omega \div u(x) \overline{p(x)} dx\\& =
\nu \langle\text{{\bf grad }}v, \text{{\bf grad} }u\rangle-\vv\langle \div v,q\rangle-\vv\langle p,\div u\rangle
 \end{aligned}\] 
 defined on 
$$
\Dom[\fs]=\{v\oplus p\, |\, v\in H_0^1(\Omega)^n, \,\, p\in L^2(\Omega)\}.
$$
Here \(\text{{\bf grad}}\) denotes the component-wise application of the standard gradient operator initially defined  on the Sobolev space \(H_0^1(\Omega)\).

Using the inequality 
\begin{equation*}\label{vlozh}
|\langle\div v, p\rangle|\le \epsilon \|(-{\bf\Delta})^{1/2}v\|^2+C(\epsilon)(\|v\|^2+\|p\|^2 )
\end{equation*}
$$
v\in  H_0^1(\Omega)^n,\quad p \in L^2(\Omega),
$$
 valid for any $\epsilon >0$, with $C(\epsilon)$ an appropriately chosen constant, 
one verifies that $\fs$  on \(\dom[\fs]\) is a closed semi-bounded form (by the KLMN-Theorem).

 We also remark  that    
the closure of the operator matrix \eqref{StokesMatrix} defined on  \((H^2(\Omega)
\cap  H_0^1(\Omega))^n\oplus H^1(\Omega)\) is a self-adjoint operator, see  \cite{FFMM}, which yields another characterization for the operator  \(S=S(\nu,\vv)\).

We now provide more detailed information on the location of the spectrum of the Stokes operator \(S\).

\begin{theorem}\label{osn2}  Let $S$ be the Stokes operator. Then
\begin{itemize}
\item[(i)]
the  positive  spectrum of $S$   is discrete and
\begin{equation*}\label{assass2}\lambda_1(S)\ge  \nu \lambda_1(\Omega),\end{equation*}
where \(\lambda_1(S)\) is the smallest positive eigenvalue of $S$ and $\lambda_1(\Omega) $ is the principal 
eigenvalue of the  Dirichlet Laplacian in $\Omega$. Moreover, 
the asymptotic representation 
\begin{equation}\label{assass}
\lambda_1(S)=\nu \lambda_1(\Omega)(1+o(1) ) \quad \text{as }\quad \nu\to \infty\quad \text{or}\quad v_*\to 0, 
\end{equation}
holds;

\item[(ii)] the point $\lambda=0$ is an isolated simple eigenvalue of $S$;

\item[(iii)] the bottom of the (essential)  spectrum  of the Stokes operator is explicitly given by 
 \begin{equation}\label{sminus}
 \inf \spec(S)=-\frac{\vv^2}{\nu}.
\end{equation}
\end{itemize}

 In particular, 
$$
\inf \spec(S)=-\frac14\nu \lambda_1(\Omega)[\rlad]^2,
$$
where
\begin{equation}\label{reyrey}
\rlad=\frac{2\vv}{\nu \sqrt{\lambda_1(\Omega)}}
\end{equation} is the generalized Reynolds number. Moreover, one has the estimate
\begin{equation}\label{eq:bottom}|\inf \spec (S)|\leq \frac14 \left [\rlad\right ]^2 \lambda_1(S).
\end{equation}
\end{theorem}

 \begin{proof}

(i).  It is well known that  the essential spectrum of the Stokes operator  $S$ is purely negative  \cite{Atkinson}, \cite{FFMM}, \cite{Grubb}, therefore,  the positive spectrum of $S$ is discrete.

The inequality
\begin{equation}\label{simsim} \lambda_1( S) \ge \nu \lambda_1(\Omega).
 \end{equation}
follows from Lemma \ref{numran} (vi) (see Appendix \ref{Appendix}).
 
To prove the asymptotics \eqref{assass}, we proceed as follows.

Let   \(\lambda_1(\Omega)\)  denote the first positive eigenvalue of the vector-valued Dirichlet problem 
$$-{\bf\Delta} f=\lambda_1(\Omega) f,$$
$$f|_{\partial \Omega}=0,$$
with $f$ the corresponding eigenfunction.  Introducing
$
 v=(f,0)^T \in \cH,
$
 one observes that 
\begin{equation}\label{trialfunction}\left\|Sv-\nu\lambda_1(\Omega)v\right\|=\vv\left\| \div f \right \|. \end{equation}

Using the standard estimate
$$
\dist (\lambda, \spec (T))\le \frac{\|(T-\lambda I)x\|}{\|x\|}, \quad x\in \Dom (T),
$$
valid for  any self-adjoint operator $T$,
from \eqref{trialfunction}  it follows that 
$$
\dist (\nu\lambda_1(\Omega), \spec (S))\le \vv \frac{\|\div f\|}{\|f\|}
.$$

This yields the claimed asymptotics \eqref{assass}  for \(\vv\to 0\) and, by rescaling, for \(\nu \to \infty\).  
	Taking into account that  the open interval  $(0, \nu \lambda_1(\Omega))$ is free of the spectrum of $S$, we even get that 
  \begin{equation}\label{twosidedestimate}
\nu\lambda_1(\Omega)\le  \lambda_1(S)\le  \nu\lambda_1(\Omega)+\vv\frac{\|\div f\|}{\|f\|}.
 \end{equation}

(ii). We claim that \[\Ker(S)=\{0\oplus p\;|\; p \text{ constant }\}\subset L^2(\Omega)^n\oplus L^2(\Omega).\]
Indeed,  by  \cite[Theorem 1.3]{SchPaper}, 
\begin{equation}\label{eqKer}\Ker(S)=(\Ker(-{\bf\Delta})\cap \cL_+)\oplus \cL_-\subset L^2(\Omega)^n\oplus L^2(\Omega),\end{equation}
where\[\cL_+=\{v\in H_0^1(\Omega)^n\;|\; \langle \div v, p \rangle=0 \;\text{ for all } p \in L^2(\Omega)\}\] and 
\begin{equation}\label{ssdd}
\cL_-=\{p\in L^2(\Omega) \;|\; \langle \div v, p \rangle=0 \;\text{ for all } v \in H_0^1(\Omega)^n\}.
\end{equation}
Since $\Ker(-{\bf\Delta})$ is trivial, we have that \(\Ker(S)=\cL_-\). This  means that the pressure \(p\) is a constant function for  \(\grad p=0\) (which is due to \eqref{ssdd}: \(\langle \div v, p \rangle=0\) for all \(v\)).

Since the essential  spectrum of the Stokes operator is purely negative, it follows that \(\lambda =0\) is an isolated eigenvalue of \(S\) of multiplicity one.

(iii). We prove \eqref{sminus} by  applying  Lemma  \ref{numran} (iv)  (see   Appendix B)
that states that 
  \begin{equation}\label{eins}
\inf \spec (S)=\inf W^2[\fs],
\end{equation}
where  
\begin{equation}\label{zwai}
W^2[\fs]=\bigcup_{\substack{v\oplus p\in H_0^1(\Omega)^n\oplus L^2(\Omega),\\ \|v\|=\|p\|=1}}\mathrm{spec}\begin{pmatrix}\nu\left\|{\bf grad}\; v\right\|^2& -\vv\overline{\langle\div v,p\rangle}\\-\vv\langle\div v,p\rangle&0\end{pmatrix} 
\end{equation}
is the quadratic numerical range (associated with the decomposition $\cH=(L^2(\Omega))^n\oplus L^2(\Omega)$).  
Consider the trial functions
\begin{equation}\label{ssddff}
 u(x_1, \dots, x_n)=f(x_1, \dots, x_n)e^{ik x_1}\begin{pmatrix}1\\0\\\vdots\\0\end{pmatrix}, \quad p(x_1, \dots, x_n)=f(x_1, \dots, x_n)e^{ik x_1},
 \end{equation}
 where \(f\in C_0^\infty(\Omega)\) with \(\left\|f\right\|=1\) ($k$ is a large parameter).
   From \eqref{eins} and \eqref{zwai} one  gets the estimate
\begin{equation}\label{ffff}
\inf \spec (S)\le  \inf \spec\begin{pmatrix}\nu\left\|{\bf grad}\; u\right\|^2& -\vv\overline{\langle\div u,p\rangle}\\-\vv\langle\div u,p\rangle&0\end{pmatrix}.
\end{equation}  Since for $u$ and $p$ given by \eqref{ssddff} we have that 
$$\nu \left\|{\bf grad}\: u \right\|^2= \nu k^2 +{\mathcal O}(k)\quad \text{and}\quad  
\vv\langle \div u, p\rangle=i\vv k+{\mathcal O}(1)\quad \text{as } k\to \infty,
$$
inequality \eqref{ffff} yields
\begin{align*}
\inf \spec (S)&\le\lim_{k \to \infty}  \inf\spec \begin{pmatrix}\nu k^2+{\mathcal O}(k)& i\vv k+{\mathcal O}(1)\\-i\vv k+ {\mathcal O}(1)&0\end{pmatrix}
\\&=
\lim_{k \to \infty}
\inf \spec\begin{pmatrix}\nu k^2& i\vv k\\-i\vv k&0\end{pmatrix} =-\frac{\vv^2}{\nu}.
\end{align*} 

To prove the opposite inequality, suppose that 
 $
v\in H_0^1(\Omega)^n$  and $ p \in  L^2(\Omega)$ are chosen in such a way that   $\|v\|=\|p\|=1$.
Then it is clearly seen that 
\begin{equation}\label{drei}
-\frac{\vv^2}{\nu}\leq\inf \mathrm{spec}\begin{pmatrix}\nu\left\|{\bf grad}\; v\right\|^2&\vv \|{\bf grad}\; v\|
\\ \vv\|{\bf grad}\; v\|
&0\end{pmatrix}
\end{equation}
$$
\quad \hskip 1cm\le  \inf\mathrm{spec}\begin{pmatrix}\nu\|{\bf grad}\; v\|^2& -\vv\overline{\langle\div v,p\rangle}\\-\vv\langle\div v,p\rangle&0\end{pmatrix}.
$$
Here, we used the inequality
\[|\langle\div v,p\rangle|\leq \left\|\div v\right\|\cdot \left\|p\right\|\leq \left\|{\bf grad}\; v\right\|,\]
and the observation that   the lowest eigenvalue   of a symmetric \(2\times2\)  matrix decreases whenever  the absolute value of its off-diagonal entries increases.
\end{proof}

\begin{remark}
Notice that the upper estimate 
\begin{equation*}
 \inf \spec (S)\le  -\frac{\vv^2}{\nu}
\end{equation*}
also follows  from the known fact  that  for any bounded domain $\Omega\subset \R^n$ whose boundary is of class $C^2$
 the essential  spectrum of the Stokes operator is a two-point set 
$$
\specess(S)=\left \{  -\frac{\vv^2}{\nu},- \frac{\vv^2}{2\nu}\right \}
$$
 (see, e.g., \cite[Theorem 3.15]{FFMM} where the corresponding result is proven for $\nu=1$, $\vv=-1$, and can be adapted to the case in question by rescaling).
\end{remark}

\begin{remark} As it follows form the proof,  inequality \eqref{eq:bottom}  is asymptotically sharp in the sense that 
 \begin{equation}\label{aposteriori}
\lim_{\nu \to \infty}\frac{|\inf \spec (S)|}{ \lambda_1(S)}=\lim_{\vv \downarrow 0}\frac{|\inf \spec (S)|}{ \lambda_1(S)}= \frac14[\rlad]^2 .
\end{equation}
\end{remark}

Our next ultimate  goal is  
to obtain  bounds  on the maximal rotation angle between the positive subspace of the Stokes operator and the positive subspace of its diagonal part.

Recall that if $P$ and $Q$  are orthogonal projections and $\Ran(Q)$ is a graph subspace with respect to the decomposition $\cH=\Ran(P)\oplus \Ran(P^\perp) $,
then  the operator angle $\Theta$  between the subspaces $\Ran(P)$ and $\Ran(Q)$ is defined to be  a unique self-adjoint operator in the Hilbert space $\cH$ 
with the spectrum in  $[0,\pi/2]$ 
such that
$$
\sin^2\Theta=PQ^\perp|_{\ran(P)}\;.
$$

Without any attempt to give a complete overview of the whole work done on pairs of subspaces and  operator angles, we mention the pioneering works \cite{Davis,Dix1,Dix2,Friedr,Halmos,Krein}.
For more recent works on operator angles and their norm estimates, we refer to \cite{AlbMot1, AlbMot2, Cuenin, Davis:Kahan, pap:2, Knyazev, KMM:1,Seelmann, T} and references therein. 

We now present our main result.

\begin{theorem}[\sc{The \(\tan 2\Theta\)-Theorem in Fluid dynamics}]\label{osn3} 

Denote by $\Theta$   the operator angle between the positive subspace  of the Stokes operator
 $\Ran\EE_S((0, \infty))$ and the subspace $\cH_+=L^2(\Omega)^n\oplus\{0\}$, the positive subspace of its diagonal part.

Then 
\begin{equation}\label{estim}
\tan 2\|\Theta\|\leq \rlad,
 \end{equation}
where $\rlad$ is the generalized Reynolds number.
 \end{theorem}
 
  \begin{proof} 
Denote by $Q$   be   the orthogonal projection from $\cH$ onto   the positive spectral  subspace \(\ran E_S((0,\infty))\) of the Stokes operator $S$ and let $P$ be the orthogonal projection onto  $\cH_+=L^2(\Omega)^n\oplus\{0\}$.

From  \cite[Theorem 3.1]{pap:2}  it follows
 that 
 \begin{equation}\label{nununu}\sin \|\Theta\|=\|P-Q\|\leq \sin\left(\frac{1}{2}\arctan \gamma \right),\end{equation}
 where 
  \begin{equation*}
\gamma=\inf_{\mu\in (0,\nu  \lambda_1(\Omega))} \sup_{\substack{v\oplus p\in H_0^1(\Omega)^n\oplus L^2(\Omega)}} \frac{2\vv|\Re \langle \div v,p\rangle|}{\nu\langle{\bf grad}\;  v, {\bf grad}\; v \rangle-\mu\left\|v\right\|^2+\mu\left\|p\right\|^2}. 
  \end{equation*}

Using the Poincar\'{e} inequality
$$
\|w\|\le \frac{1}{ \sqrt{\lambda_1(\Omega)}}\|\nabla   w\|,\quad w \in H_0^1(\Omega),
$$
and the bound
$$
\|\div \, v\|\le \|{\bf grad }\, v\|,
$$
 one then obtains that 
\begin{align}
\gamma&\leq \inf_{\mu\in (0,\nu  \lambda_1(\Omega))}  \sup_{\substack{v\oplus p\in H_0^1(\Omega)^n\oplus L^2(\Omega)}} \frac{2\vv\| {\bf grad} \,v\|\cdot \|p\|}{(\nu-(\lambda_1(\Omega))^{-1}\mu)\|{\bf grad}\; v\|^2 +\mu\left\|p\right\|^2 } \nonumber\\
&\leq \inf_{\mu\in (0,\nu  \lambda_1(\Omega))} \frac{\vv}{\sqrt{(\nu-(\lambda_1(\Omega))^{-1}\mu)\mu}}. \label{inff}
\end{align}
Since the infimum \eqref{inff}  is attained at the midpoint $\mu_{\mathrm{opt}}$ of the interval  $(0,\nu  \lambda_1(\Omega))$ with 
\begin{equation}\label{muopt}
\mu_{\mathrm{opt}}=\frac12 \nu \lambda_1(\Omega),
\end{equation}
we obtain the estimate 
\begin{equation}\label{ozz}
 \gamma \le \frac{2\vv}{ \nu\sqrt{ \lambda_1((\Omega)}}=\rlad.
 \end{equation}
The estimate \eqref{estim} now follows from   \eqref{nununu}.

 \end{proof}
 We remark that performing the spectral analysis of the Stokes operator can  essentially be reduced to the one of its principal symbol which is given by  the following $2\times 2$ numerical matrix (cf. \eqref{drei})
\begin{equation*}\label{fffff}
\fs(\nu,\vv;k)=\begin{pmatrix}
\nu k^2& i\vv k\\
-i\vv k&0
\end{pmatrix}
\end{equation*}  with the right choice for the ``wave number" \(k=\sqrt{\lambda_1(\Omega)}\), where \(\lambda_1(\Omega)\) is the principal Dirichlet eigenvalue of the Laplace operator on the domain \(\Omega\). 

 \begin{theorem}\label{xuxu}
Let  $\fs= \fs(\nu,\vv;\lambda_1(\Omega))$ be
the principal symbol 
of the Stokes operator is evaluated at the wave number \begin{equation}\label{eq:k}k=\sqrt{\lambda_1(\Omega)}.\end{equation} 

Then 
\begin{align}
1&=\lim_{v_*\downarrow 0}\frac{\inf\spec  (S)}{\inf \spec (\fs)}=  \lim_{\nu\to \infty}\frac{\inf \spec (S)}{\inf \spec(\fs)}.
\label{realVsModel}
\end{align}

 Moreover, the operator angle $\Theta$  referred  to in Theorem \ref{osn3} admits the following norm estimate
\begin{equation}\label{est}\|\Theta\|\le \theta,
\end{equation}
with $\theta$  the angle between the  eigenvectors of the  $2\times 2$ matrices
$\fs(\nu,\vv;\sqrt{\lambda_1(\Omega)})$ and $\fs(\nu,0;\sqrt{\lambda_1(\Omega)})$
corresponding to their positive eigenvalues.

\end{theorem}
\begin{proof}

Let $\lambda_-(\fs)$ and $\lambda_+(\fs)$ be the  negative  and positive  eigenvalues of the $2\times 2$ matrix \(\fs(\nu, v^*; k)\), respectively. 

It is easy to see that 
$$
\lim_{k\to \infty}\lambda_-(\fs(\nu, v_*;k))=-\frac{{v_*}^2}{\nu},
$$
and that 
\begin{equation*}\label{l-}
\lambda_-(\fs(\nu, v_*;k))=-\frac{{v_*}^2}{\nu}(1+o(1))\quad \text{as}\quad \nu \to \infty\quad \text{or}\quad v_*\to 0.
\end{equation*}

Moreover, 
$$
\lambda_+\big(\fs(\nu, v_*;\sqrt{\lambda_1(\Omega)})\big)= \nu\lambda_1(\Omega)(1+o(1))\quad  \text{as}\quad  \nu \to \infty
$$
and 
\begin{equation*}\label{l+}
\lim_{v_*\downarrow 0}\lambda_+\big(\fs(\nu, v_*;\sqrt{\lambda_1(\Omega)})\big)= \nu\lambda_1(\Omega)\quad  \text{as}\quad  v_* \to 0.
\end{equation*}

Comparing these asymptotics with the representations \eqref{assass} and \eqref{sminus} in  Theorem \ref{osn2} proves   \eqref{realVsModel}.

To prove the estimate \eqref{est}, observe  that  the rotation angle $\theta$ 
   between the positive eigensubspaces of  the \(2\times 2\) matrices
$$\fs=\begin{pmatrix}
\nu \lambda_1(\Omega)& iv_*\sqrt{\lambda_1(\Omega)}\\
-iv_*\sqrt{\lambda_1(\Omega)}&0
\end{pmatrix}
\quad \text{ and}
\quad 
\fs_0=\begin{pmatrix}
\nu \lambda_1(\Omega)&0\\
0&0
\end{pmatrix}
$$
is explicitly given by  (cf., \cite[Example 4.4]{pap:2})
\begin{equation*}
\theta=\frac12 \arctan \frac{2 v_*}{\nu\sqrt{\lambda_1(\Omega})}=\frac12 \arctan \rlad.
\end{equation*} 
 The estimate \eqref{est} follows then from Theorem \ref{osn3}.\qedhere

\end{proof}

\section{Reynolds number less than one}
\label{Sec:Motivation}
 
In this section, we discuss the case of low Reynolds number (in any dimension \(n\geq2\)).

\vspace{12pt}

First, we observe that by Theorem \ref{osn3}, the hypothesis  $\rlad<1$ implies

\begin{itemize}
\item[(i)] the lower edge  $\inf \spec (S)$ of the spectrum  of  the Stokes operator   and its first  positive eigenvalue  \(\lambda_1(S)\) satisfy the inequality  
\begin{equation}\label{stability1}
|\inf \spec (S)|<\frac14\cdot \lambda_1(S);
 \end{equation}

 \item[(ii)]
the operator angle $\Theta$ between the positive spectral subspaces  of the Stokes operator
$S=S(\nu, \vv)$ and the unperturbed diagonal operator $S_0=S(\nu,0)$ satisfies the inequality
$$
\|\Theta\| <\frac\pi 8.
$$
\end{itemize}

Moreover, by Theorem \ref{tridva} below, we also have that
\begin{itemize}
\item[(iii)]
 the Friedrichs extension  $ T $ of the block operator matrix
\begin{equation}\label{PosStokes} 
\begin{pmatrix}-\nu{\bf\Delta}-\frac12 \nu \lambda_1(\Omega)&\vv\grad\\ -\vv\div &\frac12 \nu \lambda_1(\Omega)\end{pmatrix} 
\end{equation}
is a positive definite operator. 
 \end{itemize}

 \begin{theorem}\label{tridva} 
  Let $T$ be   the Friedrichs extension of the block operator 
\eqref{PosStokes}. Then 
\begin{equation}\label{slab}
  \frac{ \nu\lambda_1(\Omega)}{2} 
 \left (1-\rlad
 \max \left \{1, \frac12 \rlad
 \right \}\right ) \le \inf \spec (T),
\end{equation}
  where
 $$
\rlad=\frac{2v_*}{\nu\sqrt{\lambda_1(\Omega)}}
 $$
 is the Ladyzhenskaya-Reynolds number. 
 
In particular, if $\rlad<1$, then  the operator $T$ is positive definite and 
  \begin{equation}\label{ner}
\frac{1}{2}\nu \lambda_1(\Omega)(1-\rlad)\leq \inf \spec(T),
 \end{equation}
 which is asymptotically sharp as \(\rlad\to 0\).
 \end{theorem}
 \begin{proof} As in the proof of Theorem \ref{osn2} (iii),
we apply Lemma  \ref{numran} 
to see that 
  \begin{equation*}
\inf \spec (T)=\inf W^2[\ft].
\end{equation*}
Here
\begin{equation*}\label{zwo}
W^2[\ft]=\bigcup_{\substack{v\oplus p\in H_0^1(\Omega)^n\oplus L^2(\Omega),\\ \|v\|=\|p\|=1}}\mathrm{spec}\begin{pmatrix}\nu\left\|{\bf grad}\; v\right\|^2-\frac12\nu \lambda_1(\Omega)& -\vv\overline{\langle\div v,p\rangle}\\-\vv\langle\div v,p\rangle&\frac12  \nu \lambda_1(\Omega)\end{pmatrix}, 
\end{equation*}
is the quadratic numerical range (associated with the decomposition $\cH=(L^2(\Omega))^n\oplus L^2(\Omega)$).

We claim that 
\begin{equation}\label{unoo}
 \frac{ \nu\lambda_1(\Omega)}{2}
 \inf_{1\le x }\inf \spec \begin{pmatrix}2x^2-1&\rlad  x \\ \rlad x&1 \end{pmatrix} \le \inf W^2[\ft].
\end{equation}

Indeed, introduce the notation
$
k=\|{\bf grad }\, v\|
$.
Then
\begin{equation*}\label{diva}
|\langle\div v,p\rangle\| \le \|{\bf grad }\, v\|\, \|p\|=k
\end{equation*}
due to the hypothesis that $\|v\|=\|p\|=1$. Therefore, the lowest eigenvalue of the \(2\times2\) matrix
$$\begin{pmatrix}\nu\left\|{\bf grad}\; v\right\|^2-\frac12\nu \lambda_1(\Omega)& -\vv\overline{\langle\div v,p\rangle}\\-\vv\langle\div v,p\rangle&\frac12  \nu \lambda_1(\Omega)\end{pmatrix}
$$
does not exceed the one of 
$$
 \begin{pmatrix}\nu k^2-\frac12 \nu \lambda_1(\Omega)& -\vv k\\ -\vv k&\frac12 \nu \lambda_1(\Omega)\end{pmatrix}.
$$
Due to  the Poincar\'e inequality, one also has the bound
\begin{equation*}\label{kbound}
\sqrt{\lambda_1(\Omega)}\le k.
\end{equation*}
 Thus,
$$
 \inf_{\sqrt{\lambda_1(\Omega)}\le k} \inf \spec   \begin{pmatrix}\nu k^2-\frac12 \nu \lambda_1(\Omega)& -\vv k\\ -\vv k&\frac12 \nu \lambda_1(\Omega)\end{pmatrix}  \le \inf W^2[\ft]
$$
and \eqref{unoo} follows.

 Next, it is an elementary exercise to see that the lowest eigenvalue of the matrix
 $$
M(x) =\begin{pmatrix}2x^2-1&-\rlad  x \\ -\rlad x&1 \end{pmatrix} 
 $$
 is a monotone function in $x$ on $[1,\infty)$ (increasing for $\rlad <2$ and decreasing for $\rlad >2$), and therefore
 $$
  \frac{ \nu\lambda_1(\Omega)}{2} \min\left \{
 1-\rlad, 1-\frac12 [{\rlad}]^2
 \right \} \le \inf W^2[\ft], $$
which proves the bound \eqref{slab}.

To show that the estimate is asymptotically sharp, denote  by \(\lambda_1(\Omega)\)   the first eigenvalue of the Dirichlet problem 
$$-{\Delta} f=\lambda_1(\Omega) f,$$
$$f|_{\partial \Omega}=0,$$
with $f$ the corresponding eigenfunction.  Introducing
$$
 v=(f, \underbrace{0, \dots 0,}_{n-1\,\,\,\text{times}}0)^T\in \cH,
$$
 one observes that 
\begin{equation}\label{rav}\|Tv-\frac12 \nu\lambda_1(\Omega)v\|=\vv\left\| f_{x_1} \right \|, \end{equation}
which implies  
$$
\dist\left  (\frac12\nu\lambda_1(\Omega), \spec (T)\right )\le \vv \frac{\|f_{x_1}\|}{\|f\|}
.
$$
Thus,
  $$
\frac12\nu\lambda_1(\Omega)\le \inf \spec(T) \le  \frac12 \nu\lambda_1(\Omega)+\vv\frac{\|f_{x_1}\|}{\|f\|} 
 $$
and hence  $$\frac{1}{2}\nu \lambda_1(\Omega)(1-\rlad)=\inf \spec(T)(1+\cO(  \rlad)) \quad \text{as} \quad \rlad\to 0,$$ which completes the proof.

 \end{proof}

\begin{remark}
The positive definiteness of the operator \(T\) is the manifestation of a  geometric variant of the Birman-Schwinger principle for off-diagonal perturbations as presented in  \cite[Corollary 3.4]{pap:2}). 

Indeed, the operator  $T_\mu=S-\mu J$, $
\mu \in (0, \nu \lambda_1(\Omega))
$, is positive definite if and only if  $$
\gamma(\mu)= \sup_{\substack{v\oplus p\in H_0^1(\Omega)^n\oplus L^2(\Omega)}} \frac{2\vv|\Re \langle \div v,p\rangle|}{\nu\langle{\bf grad}\;  v, {\bf grad}\; v \rangle-\mu\left\|v\right\|^2+\mu\left\|p\right\|^2}<1.
$$
But we have already seen (cf. \eqref{muopt} and \eqref{ozz}) that
$$
\gamma \left (\mu_{\mathrm{opt}} \right )\le \rlad<1, \qquad \text{with }\quad \mu_{\mathrm{opt}}=\frac12 \nu \lambda_1(\Omega),$$ which shows that \(T=T_{\mu_{\mathrm{opt}}}\) is positive  definite. 
\end{remark}

 \begin{remark}\label{tritri}  By \cite[Corollary 2.1]{Atkinson}, the essential spectrum of the operator matrix $T$ can be computed explicitly as
$$
  \specess(T)=\left \{\frac12 \nu \lambda_1(\Omega)-\frac{\vv^2}{\nu},
 \frac12\nu \lambda_1(\Omega)-\frac12\frac{\vv^2}{\nu}
  \right\}
 $$
 and thus
  \begin{equation}\label{specess}
\inf  \specess(T)=\frac12 \nu \lambda_1(\Omega)-\frac{\vv^2}{\nu}=\frac12 \nu \lambda_1(\Omega) \left (1-\frac12 [{\rlad}]^2\right ).
\end{equation}
In particular, it follows from  \eqref{slab} that for large values of the Reynolds number the essential spectrum of $T$ coincides with the lower edge of its spectrum, 
$$
\inf  \specess(T)=\inf  \spec(T)=\frac12 \nu \lambda_1(\Omega) \left (1-\frac12 [{\rlad}]^2\right ) \quad \text{for }\quad \rlad \ge 2,
$$ and hence  inequality \eqref{slab} turns into an  equality.
\end{remark}

We also remark   that if $\rlad<2$, then by Theorem \ref{osn2}, the spectra of the absolute value $|S|$ of the Stokes operator  restricted to the non-negative subspace of $S$ and to its orthogonal complement are subordinated (cf.  \cite[Theorem 4.2]{pap:2}).
That is, 
$$
\max \spec (|S|_{\Ran( E_S(-\infty, 0])})< \min  \spec (|S|_{\Ran( E_S(0,\infty))}).
$$

Our main motivation for the discussion of the low Reynolds number hypothesis (\(\rlad<1\))  and its implications is to better understand the functional-analytic aspects of the Ladyzhenskaya stability result that concerns the asymptotic  behavior of  solutions to   the  2D Navier-Stokes equation
$$\frac{\partial v}{\partial t}+\left\langle v,\nabla \right\rangle v-\nu {\bf\Delta} v=-\frac{1}{\rho}\grad p +f,
$$ 
$$\div v=0, \,v|_{\partial\Omega}=0,\;v|_{t=0}=v_0.
$$
Here, as usual, we are dealing 
with   a (nonstationary) flow $v$ of  an incompressible fluid that does not move close to the (smooth) boundary $\partial \Omega$ of a (bounded) domain \(\Omega\) and 
$u$, $p$, and $f$  stand for the velocity field, pressure, and the acceleration due to external  forcing, respectively.
Furthermore, \(\rho\) and \(\nu\) are the constant density and viscosity of the fluid and \(v_0\) is the initial velocity of the flow.

\begin{proposition}[{\cite[Theorem 6.5.12]{Lad}}] \label{lad}
Suppose that \(\Omega\) is a bounded domain in \(\mathbb{R}^2\) with \(C^2\)-boundary and that $v_{\text{st}}$ is a stationary solution of 
the two dimensional Navier-Stokes equation 
$$
 (v_{\text{st}}\cdot {\bf grad})v_{\text{st}}-\nu{\bf \Delta}v_{\text{st}}+\frac{1}{\rho}\grad p=f,
 $$
 $$
 \div v_{\text{st}}=0, \quad v_{\text{st}}|_{\partial \Omega}=0,
 $$ 
 such that the generalized Reynolds number  
 \begin{equation*}\label{genr}
 \rlad=\frac{2v_*}{\nu\sqrt{\lambda_1(\Omega)}}\end{equation*} is less than one, where
\begin{equation}\label{vstar}
\vv=\left (\iint\limits_\Omega\left ( \left |\frac{\partial v_{\text{st}}}
{\partial x}\right |^2+  \left |\frac{\partial v_{\text{st}}}
{\partial y}\right |^2\right )dxdy\right )^{1/2}.
\end{equation}
 
 Let $v$  be  a solution for the non-stationary problem corresponding to the same force $f$ with the initial data $v|_{t=0}\in H^2(\Omega)^2\cap{J_{0,1}}$, where \(J_{0,1}\) is the closure in $H^1$-norm of the smooth solenoidal vector fields of compact support in $\Omega$.

 Then the difference $u=v-v_{\text{st}}$ between these two solutions satisfies the inequality
$$\|u(x,t)\|\le \|u(x,0)\|\exp (-\alpha t),\quad x\in \Omega,\quad t\geq 0,$$
where\begin{equation}\label{lap}
\alpha =\nu \lambda_1(\Omega)  (1- \rlad).
\end{equation}
\end{proposition}

We remark that the Ladyzhenskaya stability hypothesis \(\rlad<1\)  of Proposition \ref{lad} implies the Stability Laws (i), (ii), and (iii). Moreover, applying  Theorem \ref{tridva} now shows that the decay exponent $\alpha$ provides a lower bound for $\inf\spec(T)$, 
 $$
\alpha =\nu \lambda_1(\Omega) (1-\rlad ) \le 2\cdot \inf \spec(T),
 $$ 
which is asymptotically sharp  in the sense that 
 $$
 \lim_{\rlad\downarrow 0}\frac{\alpha}{\inf \spec(T)}=2.
 $$
Also notice that the lowest eigenvalue of  
 the principal symbol
 $$\ft(\nu,v^\ast;k)=\begin{pmatrix}\nu k^2-\frac12 \nu \lambda_1(\Omega)& i\vv k\\ -i\vv k&\frac12 \nu \lambda_1(\Omega)\end{pmatrix}$$
   of the operator $T$ evaluated at \(k=\sqrt{\lambda_1(\Omega)}\),
 $$\ft
 = \frac12 \nu \lambda_1(\Omega)
 \begin{pmatrix}
1&i\, \rlad\\
 -i\, \rlad&1
 \end{pmatrix}
 $$ equals one half of the decay exponent  in  \eqref{lap}, that is, 
 $$
\alpha=2 \inf  \spec (\ft).
 $$

All that combined together now sheds some light on the functional analytic nature of the 2D stability in fluid dynamics.

\appendix\section{Dimensional analysis}\label{sec:Symb}

In this appendix we  present (i) a heuristic consideration that 
 motivated the particular choice of the Stokes block operator and (ii)
apply the general dimensional theory to perform  spectral analysis  of  the Stokes system.

(i). Assume that  the Navier-Stokes equation has a steady-state solution $v_{\text{st}}$ and linearize the equation  in a neighborhood of this solution to get 
$$\frac{\partial u}{\partial t}+\left\langle v_{\text{st}},\nabla \right\rangle u+\left\langle u,\nabla \right\rangle   v_{\text{st}}-\nu {\bf\Delta} u=-\frac{1}{\rho}\grad \widehat p +f,
$$ 
$$\div u=0, \,u|_{\partial\Omega}=0,\;u|_{t=0}=u_0.
$$
We assume that the field of external mass forces $f$  is time-independent.

Then,  one observes that for smooth data, the solution  $(u, \widehat p)$ of the corresponding  stationary problem satisfies the system of equations
$$
\begin{pmatrix}\left\langle v_{\text{st}},\nabla \right\rangle \cdot +\left\langle\, \cdot\, ,\nabla \right\rangle   v_{\text{st}}-\nu {\bf\Delta}  & \grad \\
-\div &0
\end{pmatrix}
\begin{pmatrix}
u\\
\frac{\widehat p}{\rho}
\end{pmatrix}=\begin{pmatrix}
f\\
0
\end{pmatrix},
$$
which can equivalently be rewritten as 
\begin{equation}\label{sys}
\begin{pmatrix}\left\langle v_{\text{st}},\nabla \right\rangle \cdot +\left\langle \,\cdot\, ,\nabla \right\rangle   v_{\text{st}}-\nu {\bf\Delta} & \vv\grad \\
-\vv \div &0
\end{pmatrix}
\begin{pmatrix}
u\\
\frac{\widehat p}{\vv\rho}
\end{pmatrix}=\begin{pmatrix}
f\\
0
\end{pmatrix}.
\end{equation}
Here we choose the parameter $\vv$ as a characteristic velocity of the stationary flow. Note that in dimension \(n=2\),  \(v_\ast\) given by \eqref{vstar} has indeed the dimension of velocity.

Neglecting the lower order terms  $\left\langle v_{\text{st}},\nabla \right\rangle u +\left\langle u ,\nabla \right\rangle   v_{\text{st}}$ in the limit $\rr\to 0$
\footnote{Introducing a 
 characteristic velocity $u^*$  of the stationary solution $u$ of the system \eqref{sys},  one observes (on a heuristic level) that the terms 
 $\left\langle v_{\text{st}},\nabla \right\rangle u$ and $ \left\langle \,\cdot\, ,\nabla \right\rangle   v_{\text{st}}$ are of the order $\vv u^*\sqrt{\lambda_1(\Omega)}$ each, meanwhile $\nu {\bf\Delta}u$ is of the order $\nu \lambda_1(\Omega)u^*$, 
 and thus
 $$
\left\langle v_{\text{st}},\nabla \right\rangle u+ \left\langle u ,\nabla \right\rangle   v_{\text{st}}=\rr \mathcal{O}(-\nu {\bf\Delta}u).
 $$
 },
  one obtains the  system
\begin{equation}\label{selim}
\begin{pmatrix}-\nu {\bf\Delta}   & \vv\grad \\
-\vv \div &0
\end{pmatrix}\begin{pmatrix}v\\q
\end{pmatrix}=
\begin{pmatrix}f\\0
\end{pmatrix},
\end{equation}
where the  ``renormalized pressure" $\widehat q$ is given by 
 $$
 \widehat q=\frac{\widehat p}{\vv\rho}.
 $$
Finally, it remains to observe that the left-hand side of \eqref{selim}  is nothing but
 the Stokes operator matrix    $S=S(\nu, \vv)$ defined by \eqref{StokesMatrixint}.

(ii). Clearly, 
  typical  physical dimensional variables associated with the steady motion \(v_{\text{st}}\) of incompressible fluid in a bounded domain  are $T$ (time), $V$ (velocity), $\nu$ (viscosity) and  $L$ (length).
Recall that in the framework of general dimensional analysis (see, e.g., \cite[\S 19]{LL}, \cite{Sedov}), given the fundamental units which are in our case the ones of length and time, to every monomial power 
$
T^\alpha V^\beta \nu^\gamma L^\delta
$  of the physical variables, one assigns
 the vector $(\alpha, \beta, \gamma, \delta)$  in a $4$-dimensional space
\begin{equation*}\label{ciorres}
T^\alpha V^\beta \nu^\gamma L^\delta\mapsto  (\alpha, \beta, \gamma, \delta)\in \R^4.
\end{equation*}
In this setting, the dimensionless quantities/monomials  form a $2$-dimensional plane $\mathbb{P}$  in $\R^4$ determined by the equations
$$\mathbb{P}\, :\, 
\begin{cases}
\alpha-\beta -\gamma=0
\\
\beta+2\gamma+\delta=0 .
\end{cases}
$$

  It is easy to see that the  two-dimensional  square lattice 
$
\Lambda=\Z^4\cap \mathbb{P}
$
has an orthogonal basis $({\bf r},{\bf s})$  associated with the dimension free variables 
\begin{equation*}\label{RSlattice}
\frac{VL}{\nu}\mapsto {\bf r}=(0, 1, -1, 1) \quad \text{and}\quad \frac{TV}{L}\mapsto {\bf s}=(1, 1,  0, -1).
\end{equation*}

That is, 
$$
\Lambda=\{m{\bf s}+n{\bf r}\, |\,  m, n \in \Z\}=\Z^4\cap \mathbb{P}.
$$

The lattice $\Lambda$ has the  square sublattice $\Lambda'$ of index $2$ (see \cite[I.2.2]{Cassels} for the definition of the index),
$$
\Lambda'=\{m{\bf s}+n{\bf r}\, |\, m=n\,\,(\text{mod } 2),  m, n \in \Z\}\subset \Lambda.
$$
In turn,  the sublattice \(\Lambda'\) has an orthogonal basis $({\bf c}, {\bf d})$ (of minimal Euclidean length) associated with the new pair of dimension free variables (see Fig. 1)
$$
\frac{TV^2}{\nu}
\mapsto {\bf c}={\bf s}+{\bf r}=(1, 2, -1, 0)
$$
and 
$$
\frac{T\nu}{L^2}
\mapsto {\bf d}={\bf s}-{\bf r}=(1, 0, 1, -2).
$$
That is,  
$$
\Lambda'=\{m{\bf c}+n{\bf d}\, |\,  m, n \in \Z\}.
$$

\begin{pspicture}(12,7)%
\psline[linestyle=dotted](3,4)(9,4)
\psline[linestyle=dotted](3,5)(9,5)
\psline[linestyle=dotted](3,6)(9,6)
\psline[linestyle=dotted](3,3)(9,3)
\psline[linestyle=dotted](3,2)(9,2)

\psline[linestyle=dotted](4,1.7)(4,6.3)
\psline[linestyle=dotted](5,1.7)(5,6.3)
\psline[linestyle=dotted](6,1.7)(6,6.3)
\psline[linestyle=dotted](7,1.7)(7,6.3)
\psline[linestyle=dotted](8,1.7)(8,6.3)

\psline[linewidth=2pt]{->}(6,4)(7,4)
\psline[linewidth=2pt]{->}(6,4)(6,5)

\psline[linestyle=dashed,linewidth=2pt]{->}(6,4)(7,5)
\psline[linestyle=dashed,linewidth=2pt]{->}(6,4)(5,5)

\psline[linestyle=dashed](4,4)(5,5)
\psline[linestyle=dashed,showpoints=true,linewidth=1pt](6,4)(7,3)
\psline[linestyle=dashed,showpoints=true,linewidth=1pt](7,5)(8,4)
\psline[linestyle=dashed,showpoints=true,linewidth=1pt](8,4)(7,3)
\psline[linestyle=dashed,showpoints=true,linewidth=1pt](6,6)(5,5)
\psline[linestyle=dashed,showpoints=true,linewidth=1pt](6,6)(7,5)

\psline[linestyle=dashed,showpoints=true,linewidth=1pt](6,4)(5,3)

\psline[linestyle=dashed,showpoints=true,linewidth=1pt](5,3)(6,2)
\psline[linestyle=dashed,showpoints=true,linewidth=1pt](6,2)(7,3)

\psline[linestyle=dashed,showpoints=true,linewidth=1pt](5,3)(4,4)

\rput(11, 6){ {\bf r}=(0, 1, -1, 1)}

\rput(11, 5){{\bf s}=(1, 1,  0, -1)}

\rput(11, 4){{\bf c}=(1, 2, -1, 0)}
 
\rput(11, 3){{\bf d}=(1, 0, 1, -2)}

\rput(6.2,5.2){{\bf s}}
\rput(7.2,4.2){{\bf r}}
\rput(4.8,5.2){{\bf d}}
\rput(7.2,5.2){{\bf c}}

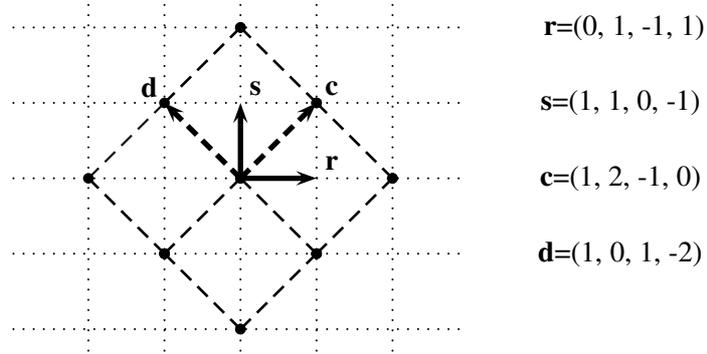
\rput(6, 1){{\bf Fig. 1.} Lattice $\Lambda$ versus sublattice $\Lambda'$ of index 2.}
\end{pspicture}

Introduce the characteristic  length scale $L\sim k^{-1}$, where  the wave number $k$ is given by \eqref{eq:k} (see Theorem \ref{xuxu}),  the characteristic velocity $\vv$, and finally 
the (characteristic) time scale $\tau$, which, in the current  setting, is at our disposal.
Then one observes that the dimensionless quantity   \(\frac{VL}{\nu}\) transforms into  
 the \textit{Reynolds}  number (cf.~ \cite[\S 19]{LL})
 \begin{equation}\label{VLnu}
 \frac{VL}{\nu}\longrightarrow
\rr=\frac{\vv}{\nu \sqrt{\lambda_1(\Omega)}}\qquad \text{as }V\longrightarrow \vv\text{ and }L\longrightarrow1/\sqrt{\lambda_1(\Omega)}.\end{equation}
In turn, the dimension-free variable  \(\frac{TV}{L}\)  gives rise to the  \emph{Strouhal type} number (cf.~ \cite[\S 19]{LL})
 \begin{equation}\label{ReSt}
\frac{TV}{L}\longrightarrow\St=\tau \big(\vv \sqrt{\lambda_1(\Omega)}\big) \qquad \text{as } T\longrightarrow \tau.
\end{equation} 
Note that the factor \( \vv \sqrt{\lambda_1(\Omega)}\) in \eqref{ReSt} has the dimension of a frequency and 
 can be interpreted as the  ``circulation frequency"  associated with the stationary flow \(v_{\text{st}}\) in the bounded domain $\Omega$. 
  
Upon the identifications above, it is striking to observe  that    the new set of dimensionless variables $\frac{TV^2}{\nu}$
and $
\frac{T\nu}{L^2}$ carries important spectral information on the Stokes operator   in the following  sense.
The  product  \(\St\rr\) is proportional to the distance from the bottom of the spectrum of the Stokes operator to the origin and the ratio  \(\St/\rr\) is proportional to the length of spectral gap of the diagonal part of the Stokes operator.
That is, 
\begin{equation*}\label{NewfreeVariables}
\frac{TV^2}{\nu}\longrightarrow
\St\rr=\tau \frac{\vv^2}{\nu}=\tau |\inf \spec (S)|
\end{equation*}
and 
 \begin{equation*}\label{NV2}
\frac{T\nu}{L^2}\longrightarrow\St/\rr=\tau\nu \lambda_1(\Omega).
\end{equation*}

  Moreover, 
from    \eqref{realVsModel}  one also derives  that 

 \begin{equation}\label{tor1}
\lim_{\St\downarrow 0}\frac{ \inf \spec (S)}{\rr\St}=\lim_{\rr\downarrow 0 }\frac{\inf \spec (S)}{\rr\St}=-\frac{1}{\tau}
\end{equation}
and that
 \begin{equation}\label{tor2}
\lim_{\St\downarrow 0}\frac{\lambda_1( S) }{\St/\rr}=\lim_{\rr\downarrow 0 }\frac{ \lambda_1( S)}{\St/\rr}=\frac1\tau.
\end{equation} 

We also notice that  the  upper bound  $\theta$ for the norm  $\|\Theta\| $ of the operator angle \eqref{est} in Theorem \ref{xuxu} can  be read off from the following diagram. 

\begin{pspicture}(12,7)%
\psline(3,4)(7.5,4)

\pswedge[linestyle=dashed,fillstyle=solid,fillcolor=lightgray](5.25,4){2.28}{90}{180}
\psline(3,4)(4.1,5.95)
\rput(1.5,6){Stability zone}

\rput(1.3,5){$
\rlad=2\rr<1
$
}

\psline{->}(2.7, 5.8)(3.3, 5.4)
\psline[linewidth=0.1pt](3.95,5.72)(4.21,5.57)
\psline[linewidth=0.1pt](4.21,5.57)(4.35,5.8)
\psarc[linestyle=dashed](5.25,4){2.25}{0}{90}
\psline(7.5,4)(4.1,5.95)
\psline(4.1,5.95)(4.1,4)
\rput(3.0,3.6){\small{$ -2  \St\rr$}}
\rput(7.5,3.6){$\frac12\St/\rr$}
\rput(4.1,3.6){0}

\rput(6.4,4.3){$2\theta$}
\rput(4.5,4.8){$\St$}
\rput(9,6){$
\theta_{\text{cr}}=\frac{\pi}{8}
$}

\rput(9,5){$
\Rey^{\ast}_{\text{crit}}=1
$}
\psarc(7.5,4){0.33}{150}{180}
\psarc(7.5,4){0.4}{150}{180}

\psline[linestyle=dashed](4.1,5.95)(9.5,4)

\psarc(9.5,4){0.33}{160}{180}
\psarc(9.5,4){0.4}{160}{180}
\rput(8.2,4.2){{\tiny$2\|\Theta\|$}}
\psline(4.1,4)(9.5,4)

\end{pspicture} 

\vskip -3cm

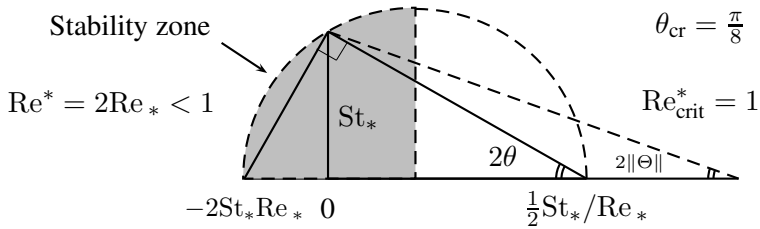
\centerline{{\bf Fig. 2.} Strouhal-Reynolds-Rotation angle  diagram. }

\vspace{12pt}

Here, the Strouhal number \(\St\), the height of the right triangle on the diagram, coincides with  
   the geometric mean of the dimensionless quantities
$ 2  \St\rr$ and $\frac12\St/\rr$. Moreover, $$
\tan 2 \|\Theta\|\le \tan 2 \theta=  \frac{\St}{\frac12 \St/\rr}=  \frac{ 2v^*}{\nu\sqrt{\lambda_1(\Omega})}=2 \, \rr=\rlad.
$$

\section{Saddle-point Forms and Numerical Ranges}\label{Appendix}

In this appendix we recall  the concept of a saddle-point form  with respect to a  decomposition 
\(\cH=\cH_+\oplus \cH_-\).

Assume that   \(A_\pm \geq 0\)
are non-negative   self-adjoint operators acting in \(\cH_\pm\). 

On $\dom[\fa]=\dom(A_+^{1/2})\oplus \dom(A_-^{1/2})\subseteq \cH$ introduce the diagonal saddle-point sesquilinear form 
\begin{equation*}\label{diaggg}
\fa[x,y]=\fa_+[x_+,y_+]-\fa_-[x_-,y_-], 
\end{equation*}
where \(x_\pm,y_\pm\in \dom[\fa_\pm]=\dom(A_\pm^{1/2})\) and 
$$
\fa_\pm[x_\pm,y_\pm]=\langle A^{1/2}_\pm x_\pm, A^{1/2}_\pm y_\pm\rangle,  \quad  x_\pm,y_\pm\in \dom[\fa_\pm]=\dom(A_\pm^{1/2}),
$$
are the non-negative closed forms associated with the self-adjoint operators $A_+$ and $A_-$, respectively.

We say that a form  $\fb$  is a saddle-point form with respect to the decomposition \(\cH=\cH_+\oplus \cH_-\)
if it admits the representation
$$\fb[x,y]=\fa[x,y]+\fv[x,y],\quad x,y\in \Dom[\fb]=\Dom[\fa],$$ where  $\fv$  is a symmetric  off-diagonal form
with respect to the decomposition, that is \[\fv[x,Jy]=-\fv[Jx,y],\] with \(J=I_{\cH_+}\oplus \left(-I_{\cH_-}\right)\). 
 
We also require that  the  off-diagonal form $\fv$ is a form bounded perturbation of the diagonal form $\fa$ in the sense that 
\begin{equation*}\label{offdiagonalestimate}
	|\fv[x]|\leq \beta(\langle |A|^{1/2}x,|A|^{1/2}x\rangle+\|x\|^2),\quad x\in \Dom[\fv] ,
\end{equation*}
for some $\beta\ge 0$.

We start with  citing the First Representation Theorem proven in   \cite[Theorem 2.7]{SchPaper}, \cite{pap:1}, \cite{pap:2} in a more general setting and adapted here to the case of saddle-point forms.
  
\begin{theorem}
\label{1repoffdiag}
Let \(\fb\) be  a saddle-point form with respect to the decomposition  \(\cH=\cH_+\oplus \cH_-\).

Then there exists a unique self-adjoint operator  \(B\) such that  $$\Dom(B)\subseteq \Dom[\fb]$$  and 
\begin{equation*}
	\fb[x,y]=\langle x,By\rangle \quad\text{ for all }\quad x\in \Dom[\fb] \quad \text{ and }\quad y\in \Dom(B).
	\end{equation*}

\end{theorem}

We say that the operator $B$ associated with the saddle-point form $\fb$  via Theorem  \ref{1repoffdiag}
satisfies the {\it domain stability condition} if
\begin{equation}\label{domain stability}
\dom[\fb]=\dom(|A|^{1/2})=\dom(|B|^{1/2}).
\end{equation}

For completeness sake, we cite the corresponding Second Representation Theorem (see \cite[Theorem 3.1]{SchPaper}, \cite{pap:1}, \cite{pap:2}).  

\begin{theorem} 
\label{second}	Let \(\fb\) be  a saddle-point form with respect to the decomposition  \(\cH=\cH_+\oplus \cH_-\) and 
 $B$ the associated operator referred to in Theorem  \ref{1repoffdiag}.
	 
	 If the domain stability condition    \eqref{domain stability} holds, 
 then the operator \(B\)  represents this form in the sense that
\begin{equation*}\fb[x,y]=\langle |B|^{1/2}x,\sign(B)|B|^{1/2}y\rangle \quad \text{for all } x,y  \in \dom[\fb]=\dom(|B|^{1/2}).
\end{equation*}
\end{theorem}

Recall  that the numerical range of an operator \(B\) is denoted as \[W(B):=\{\langle x,Bx\rangle\;|\; x\in \dom(B),\; \|x\|=1\}.\]
Accordingly, we define the \textit{numerical range} of a saddle-point form \(\fb\) as  
\[W[\fb]:=\{\fb[x]\;|\; x\in \dom[\fb],\; \|x\|=1\}.\]

Next, we generalize the concept of the quadratic numerical range for operator matrices presented in \cite{T} to the case of saddle-point forms. 

Given a saddle-point form \(\fb\) with respect to the decomposition $\cH=\cH_+\oplus \cH_-$, we define its \textit{quadratic numerical range} 
\[W^2[\fb]:=\bigcup_{\substack{x_+\oplus x_-\in \dom[\fa_+]\oplus \dom[\fa_-],\\ \|x_+\|=\|x_-\|=1}} \spec\begin{pmatrix}\fa_+[x_+]& \overline{\fv[x_+,x_-]}\\ \fv[x_+, x_-]& -\fa_-[x_-]\end{pmatrix}.\]
Here we use the standard shorthand notation \(\fa_\pm[x_\pm]=\fa_\pm[x_\pm,x_\pm]\).
 
\begin{lemma}[cf.~ \cite{T,Veselic}]\label{numran}
Let \(\fb\) be  a  saddle-point form with respect to the decomposition \(\cH=\cH_+\oplus \cH_-\) associated with the self-adjoint operator \(B\) and let \(\fa=\fa_+\oplus \left(-\fa_-\right)\) be the diagonal part of \(\fb\). 

Then
\begin{enumerate}
\item \(\spec (B)\subseteq \overline{W^2[\fb]}\);
\item \(W(B)\subseteq W[\fb]\subseteq \overline{W(B)}\subseteq \overline{(\inf \spec( B),\sup \spec (B))}\); 
\item \(W^2[\fb]\subseteq W[\fb]\);
\item \(\inf \spec (B)=\inf W^2[\fb],\quad \sup\spec ( B)=\sup W^2[\fb]\);
\item  \(W[\fa_\pm]\subseteq W^2[\fb]\) \quad if \(\dim \cH_\mp>1\);
\item if \(\fa_\pm\geq \alpha_\pm I\) for some \(\alpha_\pm\geq 0\), then \[\spec(B)\subseteq (-\infty,-\alpha_-]\cup[\alpha_+,\infty).\]

\end{enumerate}

\begin{proof} (i). First, let \(\lambda\in \mathbb{R}\) be an eigenvalue of \(B\) with corresponding eigenfunction \(u\in \dom(B)\). Since \(\dom(B)\subseteq\dom(|A|^{1/2})=\dom(A_+^{1/2})\oplus \dom(A_-^{1/2})\), we have the unique decomposition \(u=u_+\oplus u_-\) with \(u_\pm\in \dom(A_\pm^{1/2})\).
We set \(\hat{u}_\pm:=\|u_\pm\|^{-1}u_\pm\) if \(u_\pm\neq 0\) and choose \(\hat{u}_\pm\) in \(\dom(A_\pm^{1/2})\) arbitrary with \(\|\hat{u}_\pm\|=1\) if \(u_\pm=0\). From the eigenvalue equation, we obtain that \[\langle \hat{u}_+ , Bu\rangle=\lambda \langle\hat{u}_+,u_+\rangle,\quad \langle \hat{u}_-, Bu\rangle=\lambda \langle\hat{u}_-,u_-\rangle.\] By the First Representation Theorem for saddle-point forms \cite[Theorem 2.7]{SchPaper} (see also \cite{pap:1}) we can rewrite these equations in a \(2\times 2\) matrix form
\[\begin{pmatrix}\fa_+[\hat{u}_+]& \overline{\fv[\hat{u}_+,\hat{u}_-]}\\ \fv[\hat{u}_+, \hat{u}_-]&-\fa_-[\hat{u}_-]\end{pmatrix}\begin{pmatrix}\|u_+\|\\ \|u_-\|\end{pmatrix}=\lambda \begin{pmatrix}\|u_+\|\\ \|u_-\|\end{pmatrix}.\]
As a consequence, \(\lambda \in W^2[\fb]\). 

If \(\lambda \in \sigma(B)\) is not an eigenvalue, the Weyl criterion \cite[Theorem VII.12]{RS1} implies that there exists a sequence \((u^{(n)})_{n\in \mathbb{N}}\subset \dom(B)\) with \(\|u^{(n)}\|=1\) and \((B-\lambda)u^{(n)}\to 0,\;n \to \infty.\) 

In the same way as above, we write \(u^{(n)}=u^{(n)}_+\oplus u^{(n)}_-\in \dom(|A|^{1/2})\) and introduce \(\hat{u}^{(n)}_\pm\) for the normalized components. Then, we have that 
\[\langle (B-\lambda) u^{(n)},\hat{u}^{(n)}_+\oplus 0 \rangle=:v_+^{(n)},\quad \langle (B-\lambda) u^{(n)},0\oplus \hat{u}^{(n)}_- \rangle=:v_-^{(n)},\]
 both converge to zero. By the First Representation Theorem again, these equations can be rewritten as
\begin{equation}\label{cB_n}\begin{pmatrix}(\fa_+-\lambda)[\hat{u}^{(n)}_+]& \overline{\fv[\hat{u}^{(n)}_+, \hat{u}^{(n)}_+]}\\ \fv[\hat{u}^{(n)}_+, \hat{u}^{(n)}_-]&-(\fa_-+\lambda)[\hat{u}^{(n)}_-]\end{pmatrix}\begin{pmatrix}\|u^{(n)}_+\|\\ \|u^{(n)}_-\|\end{pmatrix}= \begin{pmatrix}v^{(n)}_+\\v^{(n)}_-\end{pmatrix}.\end{equation}
Let \(\cB_n-\lambda\) denote the matrix in \eqref{cB_n}. Then
\[\begin{aligned}1&=\sqrt{\|u_+^{(n)}\|^2+\|u_-^{(n)}\|^2}\\ &\leq \|(\cB_n-\lambda)^{-1}\|\cdot\sqrt{(v^{(n)}_+)^2+(v^{(n)}_-)^2}=\frac{\sqrt{(v^{(n)}_+)^2+(v^{(n)}_-)^2}}{\dist(\lambda,\spec(\cB_n))}.\end{aligned}\]
Hence \[\dist(\lambda,\spec(\cB_n))\leq \sqrt{(v^{(n)}_+)^2+(v^{(n)}_-)^2}\to 0,\; n\to \infty\] and consequently \(\lambda \in \overline{W^2[\fb]}\). 

(ii). The first inclusion \(W(B)\subseteq W[\fb]\) follows directly from the First Representation Theorem for saddle-point forms (see \cite{pap:1, SchPaper}) noting that \(\dom(B)\subseteq \dom[\fb]\).    

For the second inclusion, \(W[\fb]\subseteq \overline{W(B)}\), one can use \cite[Lemma VI.3.1]{Kato} on the form \(\fv\) to get that
\[\fb[x,x]=\langle (|A|+I)^{1/2}x,(J+R)(|A|+I)^{1/2}x\rangle-\langle x,Jx\rangle\] holds for \(x\in \dom[\fb]=\dom((|A|+I)^{1/2})\). 
Since \[B=(|A|+I)^{1/2}(J+R)(|A|+I)^{1/2}-J,\] it follows from  \cite[Theorem 2.3]{pap:1} that \(\dom(B)\) is a core for the operator \((|A|+I)^{1/2}\). The claim then is a consequence of the core property. 

The last inclusion, \(\overline{W(B)}\subseteq \overline{\big(\inf \spec (B),\sup \spec( B)\big)}\), follows directly from the well known convexity of the numerical range and statement \cite[Aufgabe VII.5.24(c)]{Wer} on the extremal points.

(iii). Let \(\lambda\in W^2[\fb]\). Then, there exist \(x_\pm\in \dom(A_\pm^{1/2})\) with \(\|x_\pm\|=1\) and \(c=(c_1,c_2)\in \mathbb{R}^2\) with \(\|c\|=1\) such that \[\begin{pmatrix}\fa_+[x_+]& \overline{\fv[x_+, x_-]}\\ \fv[x_+, x_-]& -\fa_-[x_-]\end{pmatrix}\begin{pmatrix}c_1\\c_2\end{pmatrix}=\lambda\begin{pmatrix}c_1\\c_2\end{pmatrix}.\] Taking the scalar product with \(c\) yields
\[\fb[c_1x_+\oplus c_2x_-]=\lambda.\] Since \(\|c_1x_+\oplus c_2x_-\|=1\), the claim follows. 

(iv). Note that by parts (ii) and (iii), we have \(\overline{W^2[\fb]}\subseteq \overline{\big(\inf \spec(B),\sup \spec(B)\big)}\). The claim now follows from part (i) since \(\inf \spec(B),\sup \spec(B)\in \overline{W^2[\fb]}\).

(v). Assume that \(\dim \cH_->1\). Then, for each \(x_+\in \dom[\fa_+],\; \|x_+\|=1\), there is an element \(x_-\in \dom[\fa_-],\;\|x_-\|=1\) with \(\fv[x_+, x_-]=0\). To see this, note that by \cite[Lemma VI.3.1]{Kato} \[\fv[x_+, x_-]=\langle R^\ast(A_++I)^{1/2}x_+, (A_-+I)^{1/2}x_-\rangle.\]
 Let \(f\in \cH_+\). Then, by \(\dim \cH_->1\), there exists an element \(g\in \cH_-\) such that \(\langle R^\ast f,g\rangle_{\cH_-}=0\). 

By the bijectivity of \((|A|+I)^{1/2}\colon \dom((|A|+I)^{1/2})\to \cH\), there exists a suitable \(x_-\) with \(\fv[x_+, x_-]=0\). In this case, we have that \[\fa_+[x_+]\in \spec \begin{pmatrix}\fa_+[x_+]& 0\\0&\fa_-[x_-]\end{pmatrix}\subseteq W^2[\fb].\]

(vi).
 The claim follows directly, noting that the spectrum of the \(2\times 2\) matrix 
\[\begin{pmatrix}a_+&\overline{v}\\ v&-a_-\end{pmatrix}, \quad 0\leq a_\pm<\infty ,\quad  v 
\in \C,\] is located outside of the interval \((-a_-,a_+)\). 
As a consequence, \(W^2[\fb]\cap(-\alpha_-,\alpha_+)=\emptyset\) and the claim follows. \qedhere 

\end{proof}
\end{lemma}

\subsection*{Acknowledgements}

L.~Grubi\v{s}i\'{c} was partially supported by the grant HRZZ-9345 of the Croatian Science Foundation.
 V.~ Kostrykin   is grateful to the Department of Mathematics at  the University of Missouri, Columbia, for its hospitality during his visit  
as a Miller scholar in April 2013.
K.~A.~Makarov is indebted to the Institute for Mathematics for its kind hospitality during his one month stay at the Johannes
Gutenberg-Universit\"at Mainz in the Summer of 2014. The work of K.\ A.\ Makarov has been supported in part by the Deutsche
Forschungsgemeinschaft, grant KO 2936/7-1.



\begin{thebibliography}{50}

\bibitem{AlbMot1} S.~Albeverio, A.~K.~Motovilov,
\textit{The a piori tan $\theta$ theorem for spectral subspaces}, Integ.~ Equ.~ Oper.~ Theory \textbf{73} (2012), 413 -- 430.

\bibitem{AlbMot2}S.~Albeverio, A.~K.~Motovilov,
\textit{Sharpening the Norm Bound in the Subspace Perturbation Theory}, Complex Anal.~Oper.~Theory \textbf{7} (2013) 1389 -- 1416.

\bibitem{Atkinson} F.~V.~Atkinson, H.~Langer, R. Mennicken, and A.~A.~Shkalikov,
\textit{The essential spectrum of some matrix operators}, Math. Nachr. \textbf{167} (1994), 5 -- 20.

\bibitem{Cassels} J.~W.~S.~Cassels, \textit{An Introduction to the Geometry of Numbers}, Corrected reprint of the 1971 edition, Springer-Verlag, Berlin, 1997. 

\bibitem{Cuenin}J.~-C. Cuenin, \textit{Block-Diagonalization of operators with gaps with Applications to Dirac Operators}, Reviews in Math.~ Phys.~ \textbf{24} (2012), 1250021 (31 pages).

\bibitem{Davis} C.~Davis, \textit{Separation of two linear subspaces}, Acta Scient. Math. (Szeged) \textbf{19} (1958), 172 -- 187.

\bibitem{Davis:Kahan} C.~Davis and W.~M.~Kahan, \textit{The rotation of eigenvectors by a perturbation.
III}, SIAM J. Numer. Anal. \textbf{7} (1970), 1 -- 46.


\bibitem{Dix1} J. Dixmier, \textit{Position relative de deux variet\'{e}s ferm\'{e}s dans un espace de Hilbert}, Revue Scientifique  \textbf{86}
(1948), 387 -- 399.

\bibitem{Dix2} J. Dixmier, \textit{Etude sur les vari\'{e}t\'{e}s et les op\'{e}rateurs Julia, avec quelques applications}, Bull.~ Soc. ~Math.
France \textbf{77} (1949), 11 -- 101.


\bibitem{FFMM} M.~Faierman, R.~J.~ Fries, R.~Mennicken, M.~M\"oller,
\textit{ On the essential spectrum of the linearized Navier-Stokes operator},
 Integr.~ Equ.~ Oper.~ Theory \textbf{38} (2000), 9 -- 27.


\bibitem{Friedr} K.~Friedrichs, \textit{On certain inequalities and characteristic value problems for analytic functions and for functions of two variables}, Trans.~Amer.~Math.~Soc. \textbf{41} (1937), 321 -- 364.

\bibitem{Grubb} G.~Grubb and G.~Geymonat, \textit{The essential spectrum of elliptic systems
of mixed order}, Math. Ann. \textbf{227} (1977), 247 -- 276.


\bibitem{pap:1} L.~Grubi\v{s}i\'c, V.~Kostrykin, K.~A.~Makarov, and K.~Veseli\'c, \textit{Representation theorems
for indefinite quadratic forms revisited}, Mathematika \textbf{59} (2013), 169 -- 189.

\bibitem{pap:2} L.~Grubi\v{s}i\'c, V.~Kostrykin, K.~A.~Makarov, and K.~Veseli\'c, \textit{The $\tan 2\Theta$ for indefinite quadratic forms}, J. Spectr. Theory \textbf{3} (2013), 83 -- 100.

\bibitem{Halmos} P.~Halmos, \textit{Two projections}, Trans. Amer. Math. Soc. 144 (1969), 381 -- 389. 

\bibitem{Kato} T.~Kato, \textit{Perturbation Theory for Linear
Operators}, Springer-Verlag, Berlin, 1966.

\bibitem{Knyazev} A.~Knyazev, A.~Jujunashvili and M.~Argentati, \textit{Angles between infinite dimensional subspaces with applications to the Rayleigh–Ritz and alternating
	projectors methods}, J. Funct. Anal. \textbf{259} (2010), 1323 -- 1345.

\bibitem{KMM:1}V.~Kostrykin, K.~A.~Makarov,
and A.~K.~Motovilov, \textit{Existence and
uniqueness of solutions to the operator Riccati equation. A geometric approach.} In Advances in differential equations and mathematical
physics (Birmingham, AL, 2002), vol. 327 of Contemp. Math., pages
181 -- 198. AMS, Providence, RI, 2003.
 
\bibitem{KMM:2} V.~Kostrykin, K.~A.~Makarov,
and A.~K.~Motovilov, \textit{A generalization of the $\tan 2\Theta$
theorem}, in J.~A.~Ball, M.~Klaus, J.~W.~Helton, and L.~Rodman (Eds.),
\textit{Current Trends in Operator Theory and Its Applications}, Operator
Theory: Advances and Applications Vol.~149. Birkh\"{a}user, Basel, 2004, pp.~349
-- 372.

\bibitem{Krein} M.~G.~Krein, M.~A.~ Krasnoselsky, D.~P.~Milman, \textit{On defect numbers of linear operators in Banach space and some geometric problems}, Sbornik Trudov Instituta Matematiki Nauk Ukrainskoy SSR \textbf{11} (1948), 97 -- 112. 

\bibitem{Lad} O.~A.~Ladyzhenskaya, \textit{The mathematical theory of viscous incompressible flow} Gordon and Breach, Science Publishers, New York, London, Paris, Second Edition, 1969.

\bibitem{LL} L.~ Landau and E.~ M.~ Lifshitz, \textit{Fluid Mechanics}, Pergamon Press, New York, Second Edition, 1987.

\bibitem{RS1} M.~Reed and B.~Simon, \textit{ Methods of modern mathematical physics. I. Functional Analysis}, Academic Press, Inc., 1980.

\bibitem{SchPaper} S.~Schmitz, \textit{Representation theorems for indefinite quadratic forms without spectral gap}, Integ.~ Equ.~ Oper.~ Theory \textbf{83} (2015), 73 -- 94.

\bibitem{Sedov}L.~I.~ Sedov, \textit{Similarity and Dimensional Methods in Mechanics 10th Revised ed.}, CRC Press, 1993.

\bibitem{Seelmann} A.~Seelmann, \textit{Notes on the sin2\(\theta\) theorem}, Integ.~ Equ.~ Oper.~ Theory \textbf{79} (2014), 579 -- 597.

\bibitem{T} C.~Tretter, \textit{Spectral theory of block operator matrices and applications},  Imperial College Press, London, 2008.

\bibitem{Veselic} K.~Veseli\'c, \textit{Spectral perturbation bounds for
selfadjoint operators. I}, Oper.~ Matrices \textbf{2} (2008), 307 -- 339.

\bibitem{Wer} D.~Werner, 
\textit{Funktionalanalysis 6. Auflage}, Springer, Berlin, 2007.

\end{thebibliography}
\end{document}